\documentclass[pdflatex,sn-mathphys-num]{sn-jnl}

\usepackage{graphicx}%
\usepackage{multirow}%
\usepackage{amsmath,amssymb,amsfonts}%
\usepackage{amsthm}%
\usepackage{mathrsfs}%
\usepackage[title]{appendix}%
\usepackage{xcolor}%
\usepackage{textcomp}%
\usepackage{manyfoot}%
\usepackage{booktabs}%
\usepackage{algorithm}%
\usepackage{algorithmicx}%
\usepackage{algpseudocode}%
\usepackage{listings}%

\usepackage[mathscr]{eucal}
\usepackage{relsize}
\usepackage{makecell}
\usepackage{verbatim}

\theoremstyle{thmstyleone}%
\newtheorem{theorem}{Theorem}
\newtheorem{proposition}[theorem]{Proposition}%

\theoremstyle{thmstyletwo}%
\newtheorem{example}{Example}%

\theoremstyle{thmstylethree}%
\newtheorem{definition}{Definition}%

\newtheorem{corollary}{Corollary}%
\newtheorem{lemma}{Lemma}%

\begin{document}

\title[Magnetic geodesics]{Homogeneity of magnetic geodesics in the Heisenberg group}


\author[1]{\fnm{Jun-ichi} \sur{Inoguchi}}\email{inoguchi@math.sci.hokudai.ac.jp}

\author*[2]{\fnm{Marian Ioan} \sur{Munteanu}}\email{marian.ioan.munteanu@gmail.com}

\affil[1]{\orgdiv{Department of Mathematics}, \orgname{Hokkaido University},
 \orgaddress{\city{Sapporo}, \postcode{060-0810},  \country{Japan}}}

\affil[2]{\orgdiv{Department of Mathematics}, \orgname{University  Alexandru Ioan Cuza Iasi},
 \orgaddress{\street{Bd. Carol I, n. 11}, \city{Iasi}, \postcode{700506}, \country{Romania}}}
\equalcont{These authors contributed equally to this work.}



\abstract{
We prove that magnetic geodesics in the $3$-dimensional 
Heisenberg group derived from the canonical 
contact structure are homogeneous. 
}

\date{\today}

\maketitle

\section{Introduction}
On an oriented Riemannian $3$-manifold 
$(M^3,g,dv_g)$, a (static) \emph{magnetic field} is 
realized as a divergence free vector field $B$. 
A \emph{magnetic geodesic} in $M$ under the influence 
of $B$ with charge $q$ is a curve $\gamma$ in $M$ obeying 
the \emph{Lorentz equation}
\[
\nabla_{\dot{\gamma}}\dot{\gamma}=qB\times \dot{\gamma}.
\]
Here $\nabla$ is the Levi-Civita connection of $g$ and 
$\times$ is the cross product determined by the volume element $dv_g$. 
The charge $q$ is a constant. 
In case $B=0$ or $q=0$, magnetic geodesics are 
nothing but geodesics. 
From the Lorentz equation, one can see that every magnetic geodesic is of constant speed. 
Via the metric $g$ and the volume element $dv_g$, 
a magnetic field $B$ is identified with a closed $2$-form $F$. 
Take the skew-adjoint endmorphism field $J$ 
metrically equivalent to $F$, then 
the Lorentz equation is rewritten as
\[
\nabla_{\dot{\gamma}}\dot{\gamma}=qJ \dot{\gamma}.
\]
This ODE is valid for Riemannian manifolds of arbitrary dimension. 
Based on this observation, the notion of the magnetic field 
is extended to Riemannian manifolds of arbitrary dimension as a 
closed $2$-form.

As is well known, the equation of geodesics on a Riemannian manifold 
$M$ is reformulated as a Hamiltonian system 
on the tangent bundle $TM$ of $M$ whose Hamiltonian is 
the kinetic energy $\mathrm{E}$. The solutions 
of the Hamiltonian system determined by the 
geodesic equation are called 
\emph{geodesic flows}.

By perturbing 
the canonical symplectic form $\varOmega_{TM}$ of 
$TM$ by using the magnetic field $F$ in $M$, 
one obtains a magnetized symplectic form $\varOmega_F
=\varOmega_{TM}+q\pi^{*}F$, 
where $\pi:TM\to M$ is the projection. 
One can see that the Lorentz equation is reformulated as 
a Hamiltonian system $(TM,\varOmega_F,\mathrm{E})$. 
The solutions of this Hamiltonian system are often called 
\emph{magnetic geodesic flows} (see \textit{e.g.}, \cite[Appendix C]{IM232}).
 
Magnetic geodesics have been received much attention not only from 
differential geometers but also from researchers of dynamical systems (see \textit{e.g.} \cite{AnSi, Arnold1961, Arnold1986, DGJ}) 
as well as from symplectic geometers \cite{ABM,BM}.

Particular examples of magnetic field in dimension $3$ are \emph{uniform magnetic fields} ($\nabla B=0$) 
and \emph{Killing magnetic fields} ($\pounds_{B}g=0$). 
Needless to say, the Euclidean $3$-space $\mathbb{E}^3$ is one of the model spaces of Thurston's 
$3$-dimensional geometry \cite{Thurston}.
All the eight model spaces have non trivial Killing vector fields. Indeed, 
the dimension $d$ of the Lie algebra of Killing vector fields on a model space 
satisfies $d=3$, $4$ or $6$. Moreover, the model spaces are 
represented as homogeneous Riemannian $3$-space 
$G/H$, where $G$ is the largest connected Lie group of isometries. 
In particular, except the model space $\mathrm{Sol}_3$ of \emph{solvegeometry}, 
all the model spaces are \emph{naturally reductive homogeneous $3$-spaces}.

In an oriented homogeneous Riemannian 
space $M=(G/H,g,dv_g)$ equipped with a magnetic 
field $F$, one can consider \emph{homogeneous magnetic geodesics}. 
A magnetic geodesic under the  influence of $F$ is said to be \emph{homogeneous} if it is an orbit 
of a one-parameter subgroup of $G$. Obviously, homogeneous magnetic geodesics are 
homogeneous geodesics in the sense of Kowalski-Vanhecke \cite{KV} when $F=0$ or $q=0$.  
Homogeneous Riemannian spaces, all of whose geodesics are homogeneous with respect to the largest connected group of isometries 
are called \emph{Riemannian g.~o.~spaces} \cite{KV}. 
After the publication of the seminal paper \cite{KV}, Riemannian g. o. spaces 
are investigated very actively. See 
Arvanitoyeorgos's survey \cite{Arv} and the monograph \cite{BN} due 
to Berestovskii and Nikonorov. 

Concerning magnetic geodesics on 
homogeneous Riemannian spaces, 
Bolsinov and Jovanovi{\'c} \cite{BJ} proved

{\it
 the 
homogeneity of magnetic geodesics on 
compact \emph{normal homogeneous spaces}.  }

\noindent
They showed that

{\it
 every magnetic geodesic starting at the origin under the
influence of the standard invariant magnetic field 
of compact \emph{normal homogeneous spaces}
is homogeneous }
(see \cite[Remark~1]{BJ}). Note that compact normal homogeneous spaces are naturally reductive. 
It would be interesting to generalize Bolsinov-Jovanovi{\'c}'s
result to naturally reductive homogeneous spaces, or more generally, to 
Riemannian g.~o.~spaces.

It should be noted that Bimmermann
and Maier computed the Hofer-Zender capacity of certain lens spaces 
\cite{BM}. In the work
\cite{BM}, magnetic geodesics on the lens space $L(p;1)=
\mathbb{S}^3/\mathbb{Z}_p$ induced by the 
\emph{standrad contact magnetic
field} on the unit $3$-sphere 
$\mathbb{S}^3$ play a crucial role (see also \cite{ABM}). 
Moreover, contact magnetic fields on contact $3$-manifolds are used to
construct certain spacetimes in general relativity \cite{IsMa,KIKM}.

Let us return  to Thurston geometry. 
The seven model spaces, other than $\mathrm{Sol}_3$, are 
naturally reductive, and hence they are Riemannian g.~o.~spaces. 
Thus, every geodesic in these seven model spaces is
homogeneous. 
Among those seven model spaces, 
the only normal homogeneous space is the $3$-sphere $\mathbb{S}^3$. 
The $3$-sphere is a typical example of contact 
manifold. 
As a first attempt to generalize Bolsinov-Jovanovi{\'c}'s 
result to naturally reductive homogeneous spaces, we concentrate 
our attention to the following two model spaces: the Heisenberg group $\mathrm{Nil}_3$ 
and the universal covering $\widetilde{\mathrm{SL}}_2\mathbb{R}$ 
of the special linear group $\mathrm{SL}_2\mathbb{R}$. 
These two spaces are naturally reductive and 
have standard contact structures.

The parametric expression 
of geodesics of the Heisenberg group $\mathrm{Nil}_3$ is well known (see \textit{e.g.}, \cite{IKOS1,Mare,Mu}).
More generally, geodesics of nilmanifolds 
have been studied in detail. 
See \cite{EGM,KOR,Ovando}. 
Epstein, 
Gornet and Mast \cite{EGM} 
determined periodic magnetic geodesics in compact 
Heisenberg nilmanifolds. 
Ovando and Sublis \cite{OS23,OS} studied certain 
magnetic geodesics in $2$-step nilpotent Lie groups.

The identity component of the \emph{largest} isometry group
of $\mathrm{SL}_2\mathbb{R}$ [resp. $\mathrm{Nil}_3$] is 
$\mathrm{SL}_2\mathbb{R}\times\mathrm{SO}(2)$ 
[resp. $\mathrm{Nil}_3\ltimes \mathrm{U}(1)$]. 
The Lie algebra of $\mathrm{Nil}_3\ltimes \mathrm{U}(1)$ is called the \emph{oscillator algebra}. 
These two spaces are 
represented as \emph{naturally reductive homogeneous spaces} 
$(\mathrm{SL}_2\mathbb{R}\times\mathrm{SO}(2))/\mathrm{SO}(2)$ and 
$(\mathrm{Nil}_3\ltimes \mathrm{U}(1))/\mathrm{U}(1)$. 
Recently, we proved the homogeneity of contact magnetic geodesics 
in the naturally reductive homogeneous space 
$(\mathrm{SL}_2\mathbb{R}\times\mathrm{SO}(2))/\mathrm{SO}(2)$ \cite{IM23}.
Thus, the homogeneity of contact magnetic geodesics in 
$\mathrm{Nil}_3$ would be expected.

Both  
$\mathrm{SL}_2\mathbb{R}$ and $\mathrm{Nil}_3$ have the 
naturally reductive representation 
of the form 
$(G\times K)/K$ where $G=\mathrm{SL}_2\mathbb{R}$ or $\mathrm{Nil}_3$ and 
$K$ is the rotation group 
$\mathrm{SO}(2)\cong \mathrm{U}(1)$. Although 
$K$ is a subgroup of $G$ in the case $G=\mathrm{SL}_2\mathbb{R}$, 
it is \emph{not} a subgroup of $G$ in the case $G=\mathrm{Nil}_3$.

In the proof of homogeneity of contact magnetic geodesics 
in $(\mathrm{SL}_2\mathbb{R}\times\mathrm{SO}(2))/\mathrm{SO}(2)$, 
the simplicity of $\mathrm{SL}_2\mathbb{R}$ is crucial. More precisely, 
the non-degeneracy of the Killing form is the essential tool \cite{IM23} (see \cite{IM232} for the 
compact simple counterpart of $(\mathrm{SL}_2\mathbb{R}\times\mathrm{SO}(2))/\mathrm{SO}(2)$). 
On the other hand, $\mathrm{Nil}_3$ is nilpotent and  
has vanishing Killing form. Thus we need to develop \emph{new} 
strategy for the proof of homogeneity of contact magnetic geodesics.

The purpose of the present article is to prove the homogeneity of magnetic geodesics
in the naturally reductive homogeneous space 
$\mathrm{Nil}_3=(\mathrm{Nil}_3\ltimes \mathrm{U}(1))/\mathrm{U}(1)$ with respect to the 
(standard) contact magnetic field.

\section{Static magnetism in $3$-dimensional Riemannian geometry}
\label{sec:1}

\subsection{Landau-Hall functional}
Let $(M,g)$ be a Riemannian manifold. 
Via the Riemannian metric $g$, the 
canonical symplectic form of the cotangent bundle $T^{*}M$ 
induces a symplectic form $\varOmega_{TM}$ on the tangent bundle $TM$. 
As is well known, 
the \emph{geodesic flow} is the Hamiltonian 
flow of the Hamiltonian system $(TM,\varOmega,\mathrm{E})$ whose Hamiltonian is 
the kinetic energy 
\[
\mathrm{E}(p;v)=\frac{1}{2}g_{p}(v,v).
\]
A static \emph{magnetic field} on $M$ is a closed 
$2$-form $F$ on $M$. 
The \emph{magnetic flow} is the Hamiltonian flow 
on the Hamiltonian system $(TM,\varOmega_{TM}+q\pi^{*}F,\mathrm{E})$. Here 
$\pi:TM\to M$ is the projection and $q$ is a constant (called the 
\emph{charge} or \emph{strength}). The perturbed symplectic form $\varOmega_{F}:=\varOmega_{TM}+q\pi^{*}F$ is 
called the \emph{magnetized symplectic form}.

The \emph{Lorentz force} $J$ of $F$ is an endomorphism 
field $J$ on $M$ defined by
\[
g(JX,Y)=F(X,Y).
\]
Then the orbits of magnetic flow on the base manifold $M$ takes the form:
\begin{equation}\label{eq:1.1}
\nabla_{\dot{\gamma}}\dot{\gamma}=qJ \dot{\gamma},
\end{equation}
where $\nabla$ is the Levi-Civita connection. See {\it e.g.} \cite{AnSi,Arnold1961,Arnold1986} (see also \cite[Appendix C]{IM232}).

Clearly, when $F=0$ or $q=0$, magnetic flow reduces to geodesic flow.
The ordinary differential equation \eqref{eq:1.1} is called the 
\emph{Lorentz equation}. A curve satisfying the Lorentz equation is called 
a \emph{magnetic geodesic} (or \emph{magnetic trajectory})
under the influence of the magnetic field $F$.

One can see that magnetic geodesics are of constant speed. 

Magnetic geodesics also has a \emph{Lagrangian formulation}.  
Indeed, when $F$ has a globally defined potential $1$-form $A$, 
then magnetic geodesics are characterized as a 
critical point of the \emph{Landau-Hall functional}: 
\[
\mathrm{LH}(\gamma)=
\mathrm{E}(\gamma)+q\int A(\dot{\gamma}(s))\,ds,
\]
where $\mathrm{E}$ is the kinetic energy, as before.

\subsection{A conservation law}
Assume that $M$ is $3$-dimensional and oriented by 
the volume element $dv_g$. Then a magnetic field $F$ is identified with a 
divergence free vector field $B$ 
via the correspondence
\[
F=dv_{g}(B,\cdot,\cdot).
\]
Then 
the Lorentz equation is rewritten as 
\[
\nabla_{\dot{\gamma}}\dot{\gamma}=qB\times \dot{\gamma},
\]
where $\times$ is the cross product 
induced from the volume element $dv_g$. 
For instance, a parallel vector field $B$ is a magnetic 
field which is called a \emph{uniform magnetic field}.  
Any Killing vector field is also a magnetic field 
called a \emph{Killing magnetic field}.

Killing magnetic fields satisfy the following 
conservation law:
\begin{proposition}\label{prop:conservation}
Along a magnetic geodesic $\gamma(s)$ 
under the influence of a Killing magnetic field $B$, 
the function $g(\dot{\gamma},B)$ is constant. 
\end{proposition}
Thus $g(\dot{\gamma},B)$ is a first integral of the 
Hamiltonian system $(TM,\varOmega_{F},\mathrm{E})$.


\section{The Heisenberg group}\label{sec:3}
\subsection{Heisenberg group}
Let us denote by $\omega$ the 
\emph{canonical symplectic form} 
of $\mathbb{R}^2(x,y)$, that is, $\omega=dx\wedge dy$. 
The $3$-dimensional Heisenberg group 
associated to $(\mathbb{R}^2,\omega)$ is realized as 
the Cartesian $3$-space $\mathbb{R}^3(x,y,z)$ equipped with 
the multiplication:
\[
(x_1,y_1,z_1)(x_2,y_2,z_2)=
\left(x_1+x_2,y_1+y_2,z_1+z_2+\frac{1}{2}(x_1y_2-x_2y_1)\>\right).
\]
Note that the third component of $(x_1,y_1,z_1)(x_2,y_2,z_2)$ is rewritten as
\[
z_1+z_2-\frac{1}{2}\omega((x_1,y_1),(x_2,y_2)).
\]
One can see that the Heisenberg group 
is isomorphic to the linear Lie group
\[
\left\{
\left.
\left(
\begin{array}{ccc}
1 &  x& z+(xy)/2 \\
0 & 1 &y \\
0 & 0 &1
\end{array}
\right)\>\right|\>x,y,z
\in\mathbb{R}
\right\}.
\]
The Lie algebra of the Heisenberg group 
is isomorphic to
\[
\left\{
\left.
\left(
\begin{array}{ccc}
0 &  u& w \\
0 & 0 &v\\
0 & 0 &0
\end{array}
\right)\>\right|\>u,v,w
\in\mathbb{R}
\right\}
\]
and it is called the \emph{Heisenberg algebra}. 
The basis 
\[
E_1=\left(
\begin{array}{ccc}
0 &  1& 0 \\
0 & 0 &0\\
0 & 0 &0
\end{array}
\right),
\>
E_2=\left(
\begin{array}{ccc}
0 &  0& 0 \\
0 & 0 &1\\
0 & 0 &0
\end{array}
\right),
\>
E_3=
\left(
\begin{array}{ccc}
0 &  0& 1 \\
0 & 0 &0\\
0 & 0 &0
\end{array}
\right)
\]
satisfies the commutation relations:
\[
[E_1,E_2]=E_3,
\quad 
[E_2,E_3]=
[E_3,E_1]=0.
\]
Thus, the center $\mathfrak{z}$ of the Heisenberg algebra is 
$\mathfrak{z}=\mathbb{R}E_3$. Moreover, the Heisenberg 
algebra is $2$-step nilpotent.

The vectors $E_1$, $E_2$ and 
$E_3$ induce left 
invariant vector fields:
\begin{equation}\label{eq:2.1}
E_1=\frac{\partial}{\partial x}
-\frac{y}{2}\frac{\partial}{\partial z},
\quad 
E_2=\frac{\partial}{\partial y}
+\frac{x}{2}\frac{\partial}{\partial z},
\quad 
E_3=\frac{\partial}{\partial z}
\end{equation}
on the Heisenberg group. 

\subsection{Contact structures and CR-structures}
Here we recall the notion of contact $3$-manifold
(see \textit{e.g.}, \cite{Geiges}). 

\begin{definition}
{\rm A $1$-form $\alpha$ on a $3$-manifold $M$ is 
said to be a \emph{contact form} if it satisfies 
$\alpha\wedge d\alpha\not=0$. 
}
\end{definition}
A \emph{contact structure} of a $3$-manifold $M$ is an oriented non-integrable plane field $\xi$ of $TM$, 
\textit{i.e.}, there exists a locally defined contact form $\alpha$ 
with $\mathrm{Ker}\,\alpha=\xi$ and 
$d\alpha$ agrees with the orientation of $\xi$ around any point of $M$.

A $3$-manifold $M$ equipped with a contact structure $\xi$ is called a 
\emph{contact $3$-manifold}. 
It is known that if the contact $3$-manifold $M=(M,\xi)$ is orientable, then there exists a 
globally defined contact form $\alpha$, which annihilates $\xi$.

\begin{proposition}
A contact $3$-manifold $(M,\xi)$ with a globally defined 
contact form $\alpha$ is oriented by the 
volume element $dv_{\alpha}=\alpha\wedge d\alpha$. 
There exists a unique vector field $X_{\alpha}$ 
satisfying 
\[
\alpha(X_{\alpha})=1,
\quad 
d\alpha(X_{\alpha},\cdot)=0.
\]
The vector field $X_{\alpha}$ is called the 
Reeb vector field of the contact form $\alpha$.
\end{proposition}

\begin{definition}[CR-structure]
{\rm A \emph{strongly pseudo-convex Cauchy-Riemann structure} (spc-structure, in short) of an orientable $3$-manifold 
$M$ is a pair $(\alpha,J)$ consisting of a contact form $\alpha$ and a complex structure 
$J$ on the contact structure $\xi=\mathrm{Ker}\,\alpha$ satisfying 
\[
d\alpha(X,JX)>0,\quad X\in\varGamma(\xi).
\]
An orientable $3$-manifold $M$ equipped with a spc-structure $(\alpha,J)$ is called a 
$3$-dimensional \emph{strongly pseudo-convex CR-mannifold}. On the space $\varGamma(\xi)$ of smooth sections 
of $\xi$, we introduce $L\in\varGamma(\xi^{*}\otimes\xi^{*})$ by $L(X,Y)=d\alpha(X,JY)$ and call it 
the \emph{Levi-form}. 
On a $3$-dimensional 
strongly pseudo-convex CR-mannifold $M$, the complex structure $J$ is extended to an 
endomorphism field $J$ on $M$ by $JX_{\alpha}=0$. Next, the Levi-form $L$ is extended to the 
Riemannian metric $g$ on $M$ by 
$g= L +\alpha\otimes\alpha$. The metric $g$ is called 
the \emph{Webster metric}.
}
\end{definition}

\subsection{The contact magnetic field}

The 1-form
\[
\alpha=dz+\frac{1}{2}(ydx-xdy)
\]
is a contact form of the Cartesian $3$-space $\mathbb{R}^3(x,y,z)$.
Remarkably, $\alpha$ is a \emph{left invariant contact form} on the 
$3$-dimensional Heisenberg group. Note that, since $\alpha\wedge d\alpha=dx\wedge dy\wedge dz>0$, 
the contact structure $\xi=\mathrm{Ker}\,\alpha$ is co-oriented.
The Reeb vector field is 
$X_{\alpha}:=E_3$. The contact structure $\xi=\mathrm{Ker}\,\alpha$ is called 
the \emph{standard contact structure} of the Heisenberg group. 
Let us introduce the complex $J$ of $\xi$ by 
\[
JE_1=E_2,
\quad 
JE_2=-E_1.
\]
Then $(\alpha,J)$ gives a strongly pseudo-convex CR-structure on the Heisenberg group. 
We extend $J$ to the Heisenberg group by $JE_3=0$. 
The Webster metric is left invariant and given explicitly by
\[
g=dx^{2}+dy^{2}+\left(
dz+\frac{1}{2}(ydx-xdy)
\right)^2.
\]
The basis $\{E_1,E_2,E_3\}$ is orthonormal with respect to $g$. 
In particular, the Reeb vector field is metrically dual to $\alpha$.
The Heisenberg group equipped with the Webster metric is nothing but  
the model space $\mathrm{Nil}_3$
of \emph{nilgeometry} in the sense of Thurston \cite{Thurston}. 
Hereafter we denote the Heisenberg algebra (equipped with 
the inner product above) by $\mathfrak{nil}_3$. 

The exponential map 
$\exp_{\mathfrak{nil}_3}:\mathfrak{nil}_3\to\mathrm{Nil}_3$ 
satisfies
\[
\exp_{\mathfrak{nil}_3}(X)\exp_{\mathfrak{nil}_3}(Y)=\exp_{\mathfrak{nil}_3}\left(X+Y+\frac{1}{2}[X,Y]\right),
\quad 
X,Y\in \mathfrak{nil}_3
\]
and is given 
explicitly by
\[
\exp_{\mathfrak{nil}_3}(uE_1+vE_2+wE_3)=
\left(
\begin{array}{ccc}
1 &  u& w+(uv)/2 \\
0 & 1 &v\\
0 & 0 &1
\end{array}
\right).
\]
The Levi-Civita connection $\nabla$ of the Webster metric $g$ is 
described by $\nabla_{E_1}E_1=\nabla_{E_2}E_2=\nabla_{E_3}E_3=0$ and
\[
\nabla_{E_1}E_2=
-\nabla_{E_2}E_1=\frac{1}{2}E_3,
\quad
\nabla_{E_1}E_3=
\nabla_{E_3}E_1=-\frac{1}{2}E_2,
\quad 
\nabla_{E_2}E_3=
\nabla_{E_3}E_2=\frac{1}{2}E_1.
\]%
This table of $\nabla$ implies that the Reeb vector field 
is a unit Killing vector field. Moreover, the volume element 
of the metric $g$ coincides with $dv_{\alpha}$. 
The $2$-form $F$ associated to $X_{\alpha}$ is 
$d\alpha$ and it is a left invariant magnetic field and it is called the 
\emph{contact magnetic field} of the Heisenberg group $\mathrm{Nil}_3$. The Lorentz force of $F$ coincides with $J$. 

The study of contact magnetic geodesics 
of $\mathrm{Nil}_3$ draws upon three distinct geometries: 
Thurston geometry, contact geometry, and CR geometry.

\subsection{Slant curves}
Since $X_{\alpha}$ is a Killing magnetic field, 
the conservation law (Proposition \ref{prop:conservation}) 
implies that the angle function of a 
unit speed contact magnetic geodesic and the Reeb vector field is a 
first integral of the contact magnetic geodesic flow.

Motivated by this fact, we introduce the following notion.
\begin{definition}{\rm 
Let $\gamma(s)$ be a unit speed curve in $\mathrm{Nil}_3$. 
The angle function $\theta(s)$ of the velocity 
vector field $\dot{\gamma}(s)$ and $X_{\alpha}(\gamma(s))$ is called 
the \emph{contact angle} of $\gamma(s)$. 
A unit speed curve $\gamma(s)$ is said to be a 
\emph{slant curve} if $\theta(s)$ is constant along $\gamma(s)$.}
\end{definition}

\begin{example}[Reeb flows]
{\rm 
The integral curves $\mathrm{Exp}(sX_{\alpha})$ of the Reeb 
vector field $X_{\alpha}$ of a contact $3$-manifold are called \emph{Reeb flows}. 
Reeb flows and $\mathrm{Exp}(-sX_{\alpha})$ of $\mathrm{Nil}_3$ satisfy $\sin\theta(s)=0$.
}
\end{example}

\begin{example}[Legendre curves]
{\rm 
A regular curve $\gamma(s)$ of a contact $3$-manifold is said 
to be \emph{Legendre} if it satisfies $\alpha(\dot{\gamma}(s))=0$. 
Legendre curves in $\mathrm{Nil}_3$ are characterized by 
the condition $\cos\theta(s)=0$.
}
\end{example}
The notion of a slant curve depends on both the contact form 
$\alpha$ and the metric $g$. It should be remarked that 
the notion of Reeb flow and that of Legendre curve depend only 
on $\alpha$. Both Reeb flows and Legendre curves (especially Legendre knots) are 
one of the central topics in contact geometry and contact topology (see,
\textit{e.g.}, \cite{CHHL,EF}).

\section{Homogeneous geometry of $\mathrm{Nil}_3$}
\subsection{The oscillator group}
The identity component $\mathrm{Iso}_{\circ}(\mathrm{Nil}_3)$ is 
isomorphic to the so-called \emph{oscillator group} (see \textit{e.g.}, \cite{BigRem14,Streater}):
\[
\mathrm{Osc}=\{\mathsf{M}(x,y,z,t)
\>|\>x,y,z,t
\in\mathbb{R}\},
\]
where
\[
\mathsf{M}(x,y,z,t)=\left(
\begin{array}{cccc}
1 & x\sin t-y\cos t
& x\cos t+y\sin t
& 2z\\
0 & \cos t & 
-\sin t & x\\
0 & \sin t & \cos t & y\\
0 & 0 & 0 &1
\end{array}
\right).
\]
The multiplication law of $\mathrm{Osc}$ is expressed as
\begin{align}
&\quad (x_1,y_1,z_1,e^{it_1})
(x_2,y_2,z_2,e^{it_2})\\
&=\begin{pmatrix}
x_1+x_2\cos t_1-y_2\sin t_1
\\
y_1+x_2\sin t_1+y_2\cos t_1
\\
z_1+z_2+\frac{1}{2}
\{\cos t_1(x_1y_2-x_2y_1)
+\sin t_1(x_1x_2+y_1y_2)
\}
 \\
e^{i(t_1+t_2)}
\end{pmatrix}.
\notag
\end{align}

\subsection{The oscillator algebra}

The Lie algebra $\mathfrak{osc}$ of 
$\mathrm{Osc}$ is called the 
\emph{oscillator algebra}:
\[
\mathfrak{osc}=\left\{
\left.
m(x,y,z,t):=
\left(
\begin{array}{cccc}
0 & -y &x & 2z\\
0 & 0 & -t & x\\
0 & t & 0 & y\\
0 & 0 & 0 & 0
\end{array}
\right)
\>\right|\>
x,y,z,t\in\mathbb{R}
\right\}.
\]%
The oscillator algebra is 
generated by the basis
\[
E_1=m(1,0,0,0),
\quad 
E_2=m(0,1,0,0),
\quad
E_3=m(0,0,1,0),
\quad
E_4=m(0,0,0,1).
\]
The basis $\{E_1,E_2,E_3,E_4\}$ satisfies the commutation relations:
\[
[E_1,E_2]=E_3,\quad
[E_4,E_1]=E_2,
\quad
[E_4,E_2]=-E_1.
\]
Note that 
\begin{align*}
\mathsf{M}(x,y,z,t)
=& \exp(m(x,y,z,0))
\exp(m(0,0,0,t))
\\
=& \exp(xE_1+yE_2+zE_3)\exp(tE_4)
\end{align*}
holds. Here $\exp:\mathfrak{gl}_{4}\mathbb{R}\to\mathrm{GL}_4\mathbb{R}$ is the 
matrix exponential map.

The left invariant vector fields induced from 
$E_1,E_2,E_3$ and $E_4$ are given by
\begin{align}
E_1=&\cos t\,\frac{\partial}{\partial x}
+\sin t\frac{\partial}{\partial y}
+\frac{1}{2}(x\sin t-y\cos t)\frac{\partial}
{\partial z},\notag
\\
E_2=&-\sin t\frac{\partial}{\partial x}
+\cos t\frac{\partial}{\partial y}
+\frac{1}{2}(x\cos t+y\sin t)\frac{\partial}
{\partial z},\label{eq:3.2}
\\
E_3=&\frac{\partial}{\partial z},
\quad 
E_4=\frac{\partial}{\partial t},
\notag
\end{align}
respectively. 
\subsection{Subgroups}

The Heisenberg group $\mathrm{Nil}_3$ is identified 
with 
the following Lie subgroup 
of $\mathrm{Osc}$:
\[
\left\{
\mathsf{M}(x,y,z,0)=
\left(
\begin{array}{cccc}
1 & -y
& x
& 2z\\
0 & 1 & 
0 & x\\
0 & 0 & 1 & y\\
0 & 0 & 0 &1
\end{array}
\right)\in \mathrm{Osc}
\right\}.
\]
Indeed, 
\[
\mathrm{Nil}_3\ni
(x,y,z)\longmapsto \mathsf{M}(x,y,z,0)
\in\mathrm{Osc}
\]
is a Lie group isomorphism. 
Next, the circle group $\mathrm{U}(1)$ is 
isomorphic to the subgroup
\[
\left\{
\mathsf{M}(0,0,0,t)
=\left(
\begin{array}{cccc}
1 & 0
& 0
& 0\\
0 & \cos t & 
-\sin t & 0\\
0 & \sin t & \cos t & 0\\
0 & 0 & 0 &1
\end{array}
\right)
\in \mathrm{Osc}
\right\}
\]
under the Lie group isomorphism:
\[
\mathrm{U}(1)\ni e^{it}\longmapsto \mathsf{M}(0,0,0,t).
\]
One can see that 
\begin{equation}
\mathsf{M}(x,y,z,0)\mathsf{M}(0,0,0,t)=\mathsf{M}(x,y,z,t),
\end{equation}
for all $x$, $y$, $z$, $t\in\mathbb{R}$.

The Lie algebras of these Lie subgroups are
\begin{align*}
&
\left\{
\left(
\begin{array}{cccc}
0 & -y & x & 2z\\
0 & 0 & 0 & x\\
0 & 0 & 0 & y\\
0 & 0 & 0 &0
\end{array}
\right)
=xE_1+yE_2+zE_3\in\mathfrak{osc}
\right\}\cong\mathfrak{nil}_3,
\\
& 
\left\{
\left(
\begin{array}{cccc}
0 & 0 & 0 & 0\\
0 & 0 & -t & 0\\
0 & t & 0 & 0\\
0 & 0 & 0 &0
\end{array}
\right)
=tE_4\in\mathfrak{osc}
\right\}\cong\mathfrak{u}(1).
\end{align*}
The circle 
group 
$\mathrm{U}(1)$ is a normal subgroup of $\mathrm{Osc}$. 

Note that the restrictions of left invariant vector fields 
$E_1$, $E_2$ and $E_3$ given by \eqref{eq:3.2} coincide with 
the left invariant vector fields \eqref{eq:2.1} on the 
Heisenberg group $\mathrm{Nil}_3$.

Moreover, $\mathrm{Nil}_3\cap \mathrm{U}(1)=\{\mathrm{Id}\}$.
Hence we obtain the Lie group splitting
\begin{equation}
\mathrm{Osc}=\mathrm{Nil}_3\rtimes\mathrm{U}(1).
\end{equation}
We obtain a Lie group 
isomorphism:
\[
\mathrm{Osc}\to\mathrm{Nil}_3\rtimes\mathrm{U}(1);
\quad 
\mathsf{M}(x,y,z,t)\longmapsto
(x,y,z,e^{it})
\]
with multiplication law:
\[
(a_1,b_1,c_1,e^{it_1})(a_2,b_2,c_2,e^{it_2})
=((a_1,b_1,c_1)((a_2,b_2,c_2)e^{it_1}),e^{i(t_1+t_2)}),
\]
where $(a_2,b_2,c_2)e^{it_1}$ is defined by
\[
(a_2,b_2,c_2)e^{it_1}=(a_2\cos t_1-b_2\sin t_1,a_2\sin t_1+b_2\cos t_1,c_2).
\]
\subsection{Homogeneous space representation}
The Lie group $G=\mathrm{Osc}$ acts isometrically and 
transitively on $\mathrm{Nil}_3\subset \mathrm{Osc}$ by left 
translation. The isometric action of $\mathrm{Osc}=\mathrm{Nil}_3\rtimes
\mathrm{U}(1)$ is given explicitly by
\begin{align*}
&\quad (a,b,c,e^{it})\cdot (x,y,z)
\\
&=\left(a+x\cos t-y\sin t,b+x\sin t+y\cos t,c+z+
\frac12\big(
\cos t(ay-bx)
+\sin t(ax+by)
\big)
\right).
\end{align*}
The isotropy 
subgroup $H=H_{\mathsf{M}(0,0,0,0)}$ at $\mathsf{M}(0,0,0,0)$ is 
$\mathrm{U}(1)\cong \{\mathsf{M}(0,0,0,t)\}_{t\in\mathbb{R}}$.
The isotropy algebra $\mathfrak{h}$ 
is $\mathbb{R}E_4$. The tangent space $T_{\mathsf{M}(0,0,0,0)}\mathrm{Nil}_3$ is 
the Lie algebra 
$\mathfrak{nil}_3=\mathbb{R}E_1\oplus\mathbb{R}E_2\oplus\mathbb{R}E_3$ 
of $\mathrm{Nil}_3$. 
One can see that the decomposition $\mathfrak{g}=\mathfrak{h}+\mathfrak{nil}_3$ is 
\emph{reductive}. 
However, as we will see later, 
$G/H$ is \emph{not} naturally 
reductive with respect to the 
reductive decomposition 
$\mathfrak{g}=\mathfrak{h}+\mathfrak{nil}_3$. 

\section{Homogeneous magnetic geodesics}\label{sec:7}
\subsection{Geodesics and contact magnetic geodesics in $\mathrm{Nil}_3$}
The explicit parametrization of unit speed geodesics 
in $\mathrm{Nil}_3$ is well known 
(see \textit{e.g.}, \cite{EGM,IKOS1,Mare}). For later use 
we recall the parametrization here.

Let us determine the unit speed geodesics in $\mathrm{Nil}_3$ under the 
initial condition:
\[
\gamma(0)=(x(0),y(0),z(0))=(x_0,y_0,z_0),
\quad  
\dot{\gamma}(0)=aE_1+bE_2+cE_3.
\]
The unit tangent vector is
\[
\dot{\gamma}(s)=\dot{x}(s)\frac{\partial}{\partial x}
+\dot{y}(s)\frac{\partial}{\partial y}
+\dot{z}(s)\frac{\partial}{\partial z}
=\dot{x}E_1
+\dot{y}E_2+
\left(\dot{z}+\frac{1}{2}(\dot{x}y-x\dot{y})
\right)E_3.
\]
Hence the initial tangent vector is 
\[
\dot{\gamma}(0)=
aE_1
+bE_2+
\left(\dot{z}(0)+\frac{1}{2}(ay_0-bx_0)
\right)E_3,
\quad 
a=\dot{x}(0),\>
b=\dot{y}(0).
\]
Note that the contact angle is 
given by
\[
\cos\theta(s)=\dot{z}(s)+\frac{1}{2}(\dot{x}(s)y(s)-x(s)\dot{y}(s)),
\quad 
\cos\theta(0)=c.
\]
The acceleration vector field is computed as
\begin{equation}\label{eq:accel}
\nabla_{\dot{\gamma}}\dot{\gamma}
=(\ddot{x}+\cos\theta\dot{y})E_1
+(\ddot{y}-\cos\theta\dot{x})E_2+
(\cos\theta)^{\cdot}E_3.
\end{equation}
Hence the geodesic equation is equivalent to the system:
\[
\ddot{x}+\cos\theta\,\dot{y}=0,
\quad 
\ddot{y}-\cos\theta\,\dot{x}=0,
\quad 
\frac{d}{ds}(\cos\theta)=0.
\]
The third equation implies that the contact angle of a unit speed 
geodesic is constant. This fact in nothing but the consequence of the conservation 
law (Proposition \ref{prop:conservation}).

From the first and second equation, we get
\[
\dot{x}(s)=a\cos(cs)-b\sin(cs),
\quad 
\dot{y}(s)=a\sin(cs)+b\cos(cs),\quad c=\cos\theta.
\]
First we assume that $c\not=0$. 
Then, under the initial condition 
$x(0)=x_0$ and $y(0)=y_0$, we obtain
\[
x(s)=\frac{1}{c}\left(
(cx_0-b)+a\sin(cs)+b\cos(cs)
\right),
\]
\[
y(s)=\frac{1}{c}\left(
(cy_0+a)-a\cos(cs)+b\sin(cs)
\right).
\]
Since 
\[
x=\left(x_0-\frac{b}{c}\right)+\frac{1}{c}\dot{y},
\quad 
y=\left(y_0+\frac{a}{c}\right)-\frac{1}{c}\dot{x},
\]
we have
\[
\dot{x}y=\left(y_0+\frac{a}{c}\right)\dot{x}-\frac{1}{c}\dot{x}^2,
\quad x\dot{y}=\left(x_0-\frac{b}{c}\right)\dot{y}+\frac{1}{c}\dot{y}^2.
\]
As $\gamma$ is unit speed, we obtain
\[
\dot{x}y-x\dot{y}
=\left(y_0+\frac{a}{c}\right)\dot{x}
-\left(x_0-\frac{b}{c}\right)\dot{y}
-\frac{1-c^2}{c}.
\]
From these, we have 
\[
\dot{z}=c-\frac{1}{2}(\dot{x}y-x\dot{y})
=c-\frac{1}{2}\left(
\left(y_0+\frac{a}{c}\right)\dot{x}
-\left(x_0-\frac{b}{c}\right)\dot{y}
-\frac{1-c^2}{c}
\right).
\]%
Hence
\[
z(s)=z_0+cs+\frac{(1-c^2)s}{2c}-\frac{1}{2}\left(
\left(y_0+\frac{a}{c}\right)(x(s)-x_0)
-\left(x_0-\frac{b}{c}\right)(y(s)-y_0)
\right).
\]
Let us choose
$x_0=y_0=z_0=0$, then we have

\begin{align*}
x(s)=&\frac{1}{c}\left(
a\sin(cs)+b(\cos(cs)-1)
\right),
\\
y(s)=&\frac{1}{c}\left(
a(1-\cos(cs))+b\sin(cs)
\right),
\\
z(s)=&\frac{1+c^2}{2c}\,s-\frac{a^2+b^2}{2c^2}\sin(cs).
\end{align*}
Next, we consider the case $c=0$. 
In this case $\gamma(s)$ is a \emph{Legendre 
geodesic}. 
The $x$-coordinate and $y$-coordinate are 
given by
\[
x(s)=as+x_{0},\quad
y(s)=bs+y_{0}.
\]
For the $z$-coordinate we get, successively 
\[
\dot{z}(s)=-\frac{1}{2}(\dot{x}(s)y(s)-x(s)\dot{y})(s))
=-\frac{1}{2}(ay_0-bx_0),
\]
which implies
\[
z(s)=z_{0}-\frac{ay_0-bx_0}{2}s.
\]
If we choose $x_0=y_0=z_0=0$, then
$z(s)=0$.

\begin{proposition}\label{prop:geo}
The unit speed geodesic $\gamma(s)$ of $\mathrm{Nil}_3$ starting at 
the origin $(0,0,0)$ 
with initial velocity 
$\dot{\gamma}(0)=aE_1+bE_2+cE_3$ 
is parametrized as
\begin{enumerate}
\item If $c=\cos\theta\not=0$, then
\begin{equation}\label{eq:Geo}
\left\{
\begin{array}{l}
x(s)=\frac{1}{c}\left(
(a\sin(cs)+b(\cos(cs)-1)
\right),
\\
\vspace{0.15cm}
\\
y(s)=\frac{1}{c}\left(
a(1-\cos(cs))+b\sin(cs)
\right),
\\
\vspace{0.15cm}
\\
z(s)=\frac{1+c^2}{2c}\,s-\frac{a^2+b^2}{2c^2}\sin(cs), 
\end{array}
\right.
\end{equation}
where $a^2+b^2=\sin^2\theta$. 
In particular, when $c=\pm 1$, then 
$\gamma(s)=(0,0,s)$ is the Reeb flow up to orientation. 
\item If $c=\cos\theta=0$, then
\begin{equation}\label{eq:GeoL}
x(s)=as,\quad
y(s)=bs,\quad
z(s)=0, \quad a^2+b^2=1.
\end{equation} 
\end{enumerate}
\end{proposition}
\subsection{Contact magnetic geodesics}
Next we consider the Lorentz equation 
$\nabla_{\dot{\gamma}}\dot{\gamma}=qL\dot{\gamma}$ with respect to the 
contact magnetic field $F=d\alpha$. 
Assume that the contact magnetic geodesic is \emph{arc length 
parametrized}, then the Lorentz equation is 
the system:
\[
\ddot{x}+(q+\cos\theta)\,\dot{y}=0,
\quad 
\ddot{y}-(q+\cos\theta)\,\dot{x}=0,
\quad 
(\cos\theta)^{\cdot}=0.
\]
Set $c_q:=q+\cos\theta$, then we have the following result.
\begin{proposition}[\cite{DIMN}]\label{prop:mag}
The unit speed contact magnetic 
geodesic $\gamma(s)$ of $\mathrm{Nil}_3$ starting at $(x_0,y_0,z_0)$ 
with initial velocity 
$\dot{\gamma}(0)=aE_1+bE_2+cE_3$ 
is parametrized as
\begin{enumerate}
\item If $c_q\not=0$, then 
\begin{align}
x(s)=&\frac{1}{c_q}\left(
(a\sin(c_qs)+b(\cos(c_qs)-1)\,
\right)+x_{0},
\nonumber
\\
y(s)=&\frac{1}{c_q}\left(
a(1-\cos(c_qs)\,)+b\sin(c_qs)
\right)+y_{0},
\label{eq:mag}
\\
&+\frac{1}{2c_q}\left(
(ax_0+by_0)(1-\cos(c_qs))+\big(bx_0-ay_0-\frac{a^2+b^2}{c_q}\big)\sin(c_qs)
\right)
\nonumber
\end{align}
\item If $c_q=0$, then 
\begin{equation}\label{eq:mag0}
x(s)=as+x_{0},\quad 
y(s)=bs+y_{0},\quad 
z(s)=-qs+z_{0}-\frac{1}{2}(ay_0-bx_0)s.
\end{equation}
\end{enumerate}
\end{proposition}

\subsection{Homogeneous geodesics} Here we recall some basic facts on 
homogeneous geometry for our use. 

Let $M=(G/H,g)$ 
be a homogeneous Riemannian space with reductive decomposition 
$\mathfrak{g}=\mathfrak{h}+\mathfrak{m}$. 

We identify the tangent space $T_{o}M$ at the origin 
$o=H$ with the linear subspace $\mathfrak{m}$ (called the 
\emph{Lie subspace}). 
Let us introduce a symmetric bilinear map 
$\mathsf{U}_{\mathfrak{m}}:\mathfrak{m}\times\mathfrak{m}\to\mathfrak{m}$ by
\begin{equation}\label{eq:U-tensor}
2\langle \mathsf{U}_{\mathfrak{m}}(X,Y),Z\rangle
=\langle X,[Z,Y]_{\mathfrak{m}}\rangle
+\langle Y,[Z,X]_{\mathfrak{m}}\rangle,
\quad X,Y,Z\in\mathfrak{m},
\end{equation}
where the subscript $\mathfrak{m}$ means the 
$\mathfrak{m}$-component of vectors.
Then, under the identification $T_{o}M=\mathfrak{m}$, we get
\[
\nabla_{X}Y=\mathsf{U}_{\mathfrak{m}}(X,Y)+\frac{1}{2}[X,Y]_{\mathfrak{m}},
\quad X,Y\in\mathfrak{m}.
\]

\begin{definition}
{\rm 
A reductive homogeneous 
Riemannian space $M=G/H$ with 
reductive decomposition 
$\mathfrak{g}=\mathfrak{h}+\mathfrak{m}$ is said to be 
\emph{naturally reductive} with respect to $\mathfrak{m}$ 
if it satisfies $\mathsf{U}_{\mathfrak{m}}=0$.
}
\end{definition}
A homogeneous Riemannian space $M=G/H$ is said to be 
a \emph{naturally reductive homogeneous space} if it admits a 
Lie subspace $\mathfrak{m}$ such that the reductive decomposition
$\mathfrak{g}=\mathfrak{h}+\mathfrak{m}$ is naturally reductive 
with respect to it.

\begin{proposition}
A homogeneous Riemannian space $G/H$ with reductive 
decomposition $\mathfrak{g}=\mathfrak{h}+\mathfrak{m}$ 
is naturally reductive if
and only if the geodesic through 
the origin $o$ and tangent to $X
\in \mathfrak{m}=T_{o}M$ is the curve
$\exp(sX)\cdot o$, 
orbit of the one-parameter subgroup $\exp (sX)$ of $G$, 
for all $X$.
\end{proposition}

As a generalization of the class of naturally reductive homogeneous spaces, 
the notion of Riemannian g. o. space is introduced \cite{KV}.
\begin{definition}{\rm
Let $M=G/H$ be a homogeneous Riemannian space with reductive 
decomposition $\mathfrak{g}=\mathfrak{h}+\mathfrak{m}$.  
A curve $\gamma$ starting at the origin $o\in M=G/H$ is 
said to be \emph{homogeneous} if it has the form 
$\gamma(s)=\exp(sW)\cdot o$ for some vector $W\in \mathfrak{g}$.
}
\end{definition}
A reductive homogeneous Riemannian space 
$M$ is said to be a \emph{Riemannian g.~o.~spaces} if 
all of whose geodesics are homogeneous with respect to the largest connected group $G$ of isometries. 
Clearly, naturally reductive homogeneous spaces are Riemannian g.o.~spaces.

For any vector $X\in\mathfrak{g}$, 
we can introduce a 
Killing vector field
$X^{\sharp}$ on the homogeneous Riemannian space $M=G/H$ by 
\[
X^{\sharp}_{p}=\frac{d}{ds}\biggr\vert_{s=0}
\exp_{\mathfrak{g}}(sX)\cdot p,
\quad p\in M,
\]
where $\exp_{\mathfrak{g}}:\mathfrak{g}\to G$ is the exponential map.
Moreover, we have
\[
g_{p}(\nabla_{X^{\#}}X^{\#},Z^{\#})
=\langle
[X_{\mathfrak h},X_{\mathfrak m}]+\mathsf{U}_{\mathsf m}(X_{\mathfrak m},X_{\mathfrak m}),
Z\rangle
\]
for all $Z\in\mathfrak{m}$. Here the subscript 
${}_\mathfrak{h}$ means the 
$\mathfrak{h}$-part of vectors.

For any vector $X\in\mathfrak{m}$, 
$\exp_{\mathfrak{g}}(sX)\cdot o$ is a pre-geodesic starting 
at the origin $o$ 
if and only if $\mathsf{U}_{\mathfrak{m}}(X,X)=kX$ for some 
constant $k$. This is equivalent to 
\[
\langle X,[V,X]\rangle=k\langle X,V\rangle
\]
for any $V\in\mathfrak{m}$. This simple observation is 
generalized to the following (well known)
criterion (see \textit{e.g.}, \cite{Arv,BN,KV}):
\begin{lemma}
In a reductive homogeneous Riemannian space 
$M=G/H$ with reductive decomposition 
$\mathfrak{g}=\mathfrak{h}+\mathfrak{m}$, 
a curve $\gamma(s)=\exp_{\mathfrak g}(sW)\cdot o$ 
is a geodesic for some parameter $s$ 
if and only if $W\in\mathfrak{g}$ satisfies
\begin{equation}
\label{eq:homgeo}
\langle [W,V]_{\mathfrak m},W_{\mathfrak m}
\rangle=k\langle 
W_{\mathfrak m},V\rangle
\end{equation}
for all $V\in\mathfrak{m}$. Here $k$ is a constant. 
If $k=0$, then $s$ is an affine parameter.
\end{lemma}

\subsection{Homogeneous geodesics in $\mathrm{Nil}_3$}
Let us consider the Heisenberg group 
$\mathrm{Nil}_3=\mathrm{Osc}/\mathrm{U}(1)$ with 
reductive decomposition 
$\mathfrak{osc}=\mathfrak{u}(1)+\mathfrak{nil}_3$. 
As pointed out by \cite[Theorem 7.1]{TV},
the reductive decomposition 
$\mathfrak{osc}=\mathfrak{u}(1)+\mathfrak{nil}_3$ is 
\emph{not} naturally reductive. For completeness here we confirm this fact. 
With respect to the Lie subspace $\mathfrak{nil}_3$, the 
bilinear map $\mathsf{U}_{\mathfrak{nil}_3}$ define by \eqref{eq:U-tensor} is computed as (\cite{IKOS1}):
\[
\mathsf{U}_{\mathfrak{nil}_3}(E_1,E_3)=
\mathsf{U}_{\mathfrak{nil}_3}(E_3,E_1)=
-\frac{1}{2}E_{2},
\quad 
\mathsf{U}_{\mathfrak{nil}_3}(E_2,E_3)=
\mathsf{U}_{\mathfrak{nil}_3}(E_3,E_2)=
\frac{1}{2}E_1.
\]
For other combination of $i$, $j$, $\mathsf{U}_{\mathfrak{nil}_3}(E_i,E_j)=0$.
Hence $\mathfrak{osc}=\mathfrak{u}(1)+\mathfrak
{nil}_3$ is \emph{not} naturally reductive.

Represent a vector $W\in\mathfrak{osc}$ as
\[
W=W_1E_1+W_2E_2+W_3E_3+W_4E_4\in\mathfrak{osc}.
\]
Then we have
\[
[W,E_1]
=W_4E_2-W_2E_3,
\quad 
[W,E_2]
=-W_4E_1+W_1E_3,
\quad 
[W,E_3]=0,
\quad 
[W,E_4]=W_2E_1-W_1E_2.
\]
Hence the criterion is
\[
-W_2(W_3-W_4)=kW_1,
\quad 
W_1(W_3-W_4)=kW_2,
\quad
kW_3=0.
\]
In case $k=0$, we have 
\[
W_1(W_3-W_4)=0,
\quad
W_2(W_3-W_4)=0.
\]
It follows that
$W_4=W_3$ or $W_1=W_2=0$. 
Thus $W$ has one of the two forms:\\
either
\[
W=W_1E_1+W_2E_2+W_3(E_3+E_4),
\]
or 
\[
W=W_3E_3+W_4E_4.
\]
In the latter case, 
\[
\exp_{\mathfrak g}(sW)\cdot o=
\exp_{\mathfrak g}(s(W_3E_3))
\exp_{\mathfrak g}(s(W_4E_4))\cdot o=
\exp_{\mathfrak g}(s(W_3E_3))
\cdot o,
\]
since $W_4E_4$ is an element 
of the isotropy algebra at $o$. 

Next, in case $k\not=0$, we have 
$W_3=0$ and 
\begin{equation}\label{eq:7.2}
W_2W_4=kW_1,
\quad 
W_1W_4=-kW_2.
\end{equation}
From this we get
\begin{equation}\label{eq:7.3}
W_1(k^2+W_4^2)=0,\quad
W_2(k^2+W_4^2)=0.
\end{equation}
Since $k\neq0$, $k^2+W_4^2$ cannot be zero. This implies $W_1=W_2=0$.
Thus $W$ has the form $W=W_4E_4$.

As a conclusion, $\exp_{\mathfrak{osc}}(sW)\cdot o$ is a homogeneous 
geodesic in $\mathrm{Nil}_3$ up to reparametrization 
if and only if 
\[
W=W_1E_1+W_2E_2+W_3(E_3+E_4),\quad  W=W_3E_3+W_4E_4,
\quad \mbox{or}
\quad 
W=W_4E_4.
\]
In the third case, 
$\exp_{\mathfrak{osc}}(sW)\cdot o=o$, since $W$ is an element 
of the isotropy algebra at $o$. In the second case $W_3\neq0$.
In the first case, we know that 
$\exp_{\mathfrak{osc}}(s\{W_1E_1+W_2E_2+W_3(E_3+E_4)\})$ is a 
geodesic with affine parameter $s$ starting at the origin. 
This fact implies that 
$\mathrm{Nil}_3$ with the reductive 
decomposition 
$\mathfrak{osc}=\mathfrak{u}(1)+\mathfrak{m}$, where 
\[
\mathfrak{m}=\mathrm{span}\{E_1,E_2,E_3+E_4\}
\]
is a Riemannian g.~o.~space.

\begin{proposition}
In the reductive homogeneous Riemannian space 
$\mathrm{Nil}_3=\mathrm{Osc}/\mathrm{U}(1)$ with reductive 
decomposition $\mathfrak{osc}=\mathfrak{u}(1)+\mathfrak{nil}_3$,
$\exp_{\mathfrak{osc}}(sW)\cdot o$ is a pre-geodesic 
when and only when
\[
W=W_1E_1+W_2E_2+W_3(E_3+E_4),\quad 
\mbox{or} \quad W=W_3E_3+W_4E_4, \quad W_3\neq0.
\]
\end{proposition}

Here we give an alternative proof of the following 
well-known fact. 
\begin{corollary}[\cite{Ma}]
\label{cor:Ma}
The homogeneous 
Riemannian space $\mathrm{Nil}_3=
\mathrm{Osc}/\mathrm{U}(1)$ with reductive decomposition 
$\mathfrak{osc}=\mathfrak{u}(1)+\mathfrak{m}$ is naturally reductive.
\end{corollary}
\begin{proof}
Let us choose $\mathfrak{m}$ as the Lie subspace. 
Then any vector $W\in\mathfrak{m}$ has the form
$W=W_1E_1+W_2E_2+W_3(E_3+E_4)$.
Then, as we saw before, $\exp_{\mathfrak{osc}}(sW)\cdot o$ is a
geodesic. Thus $\mathrm{Osc}/\mathrm{U}(1)$ is a Riemannian g.~o.~space. 
Since $\dim \mathrm{Osc}/\mathrm{U}(1)=3$, Riemannian g.~o.~spaces are 
naturally reductive \cite{KV}.
\end{proof}
Note that one can check that the $U$-tensor $\mathsf{U}_{\mathfrak m}$ 
of the Lie subspace $\mathfrak{m}=\mathrm{span}\{E_1,E_2,E_3+E_4\}$ vanishes
by direct computation.

\begin{proposition}
For any nonzero vector $W=W_1E_1+W_2E_2+W_3E_3\in\mathfrak{nil}_3$,
the curve $
\gamma(s)=\exp_{\mathfrak{nil}_3}(sW)$ is a geodesic with affine parameter 
in $\mathrm{Nil}_3$ when and only when 
$W$ has the form
\[
W=W_1E_1+W_2E_2,\quad \mbox{or}\quad W=W_3E_3.
\]
In particular $\exp_{\mathfrak{nil}_3}(sW)$ is a unit speed geodesic 
if and only if 
\[
W=W_1E_1+W_2E_2,\quad W_1^2+W_2^2=1, \quad \mbox{or}
\quad W=\pm E_3.
\]
In the former case, $\gamma(s)$ is a Legendre geodesic. 
The geodesic $\gamma(s)$ is the Reeb flow in the latter case. 
\end{proposition}
\begin{proof}
The curve $\exp_{\mathfrak{nil}_3}(sW)$ is 
geodesic if and only if $\mathsf{U}_{\mathfrak{nil}_3}(W,W)=0$ (see \cite{IKOS1}).
Since 
\[
\mathsf{U}_{\mathfrak{nil}_3}(W,W)=-\frac{W_1W_3}{2}E_2+\frac{W_2W_3}{2}E_1,
\]
$\exp_{\mathfrak{nil}_3}(sW)$ is a geodesic with affine parameter 
$s$ if and only if $W_3=0$ (Legendre) or $W_1=W_2=0$.
\end{proof}

\begin{theorem}\label{thm:Geodesic}
Every unit speed geodesic 
$\gamma(s)$ starting at the origin in $\mathrm{Nil}_3$ is 
$\mathrm{Osc}$-homogeneous and represented by
\[
\gamma(s)=\exp_{\mathfrak{osc}}(sW)\cdot o
\]
for some $W\in\mathfrak{m}$.
\end{theorem}
\begin{proof}
Let $\gamma(s)=(x(s),y(s),z(s))$ be a unit speed geodesic 
in $\mathrm{Nil}_3$ starting at the origin $(0,0,0)$ with 
initial velocity $\dot{\gamma}(0)=aE_1
+bE_2+cE_3$. In case $c\not=0$, 
the parametric equation 
of $\gamma(s)$ is given by
\eqref{eq:Geo}. 

On the other hand, take a vector $W=W_1E_1+W_2E_2+W_3(E_3+E_4)$ with 
$W_3\not=0$, then the orbit 
$\exp_{\mathfrak{osc}}(sW)\cdot o$ is 
computed by
\[
\exp_{\mathfrak{osc}}(sW)\cdot o=\exp(sW)\mathsf{M}(0,0,0,-W_3s).
\]
Note that $\exp_{\mathfrak{osc}}$ denotes the 
exponential map $\exp_{\mathfrak{osc}}:
\mathfrak{osc}\to \mathrm{Osc}$. 
On the other hand $\exp$ of the right hand side 
denotes the matrix exponential map  
$\exp:\mathfrak{gl}_4\mathbb{R}\to\mathrm{GL}_4\mathbb{R}$.

The coordinate functions of $\exp_{\mathfrak{osc}}(sW)\cdot o$ are 
given by
\begin{equation}\label{eq:homGeo}
\left\{
\begin{array}{l}
x(s)=\frac{1}{W_3}
(-W_2+W_1\sin(W_3s)+W_2\cos(W_3s)\,),
\\
y(s)=\frac{1}{W_3}
(W_1-W_1\cos(W_3s)+W_2\sin(W_3s)\,),
\\
z(s)=
\frac{1}{2W_3^2}\left(
-(W_1^2+W_2^2)\sin(W_3s)+(W_3s)
(W_1^2+W_2^2+2W_3^2)\,
\right).
\end{array}
\right.
\end{equation}
If we change the notation as
\[
W_1=a,\quad W_2=b,\quad W_3=c=\cos\theta,
\]
then the system \eqref{eq:homGeo} coincides with \eqref{eq:Geo}.

Next, let $\gamma(s)$ be a unit speed Legendre geodesic. 
Then $\gamma(s)$ is parametrized as \eqref{eq:GeoL}. 
On the other hand, set $W=aE_1+bE_2$, then we get
\[
\exp_{\mathfrak{osc}}(sW)=(as,bs,0)=\gamma(s).
\]
Thus every unit speed geodesic starting at the origin is 
$\mathrm{Osc}$-homogeneous.
\end{proof}

\subsection{Homogeneous magnetic geodesics}
For contact magnetic geodesics, we need to compute 
the Lorentz equation for 
$\gamma(s)=\exp_{\mathfrak{osc}}(sW)\cdot o$ and deduce the criterion.

For any vector $W=W_{\mathfrak h}+W_{\mathfrak m}\in\mathfrak{osc}=\mathfrak{u}(1)+\mathfrak{m}$, 
$\gamma(s)=\exp_{\mathfrak{osc}}(sW)\cdot o$ is a contact magnetic 
geodesic if and only if 
\[
[W_{\mathfrak h},W_{\mathfrak m}]
=qJ W_{\mathfrak m},
\]
since $\mathrm{Osc}/\mathrm{U}(1)$ is naturally reductive with respect to 
$\mathfrak{m}$.
 
Let us determine all the homogeneous contact magnetic geodesics.
Note that for $W=W_1E_1+W_2E_2+W_3E_3+W_4E_4$, we have the splitting 
$W=W_{\mathfrak h}+W_{\mathfrak m}$, where 
\[
W_{\mathfrak m}=W_1E_1+W_2E_2+W_3(E_3+E_4),
\quad 
W_{\mathfrak h}=(W_4-W_3)E_4.
\]
Hence
\[
[W_{\mathfrak h},W_{\mathfrak m}]
=[(W_4-W_3)E_4,W_1E_1+W_2E_2+W_3(E_3+E_4)]
=(W_4-W_3)(-W_2E_1+W_1E_2).
\]
On the other hand,
\[
J W_{\mathfrak m}=
-W_2E_1+W_1E_2.
\]
Hence the Lorentz equation is the system
\[
-W_2(W_4-W_3)=-qW_2,
\quad 
W_1(W_4-W_3)=qW_1,
\]
equivalently,
\[
W_1\{q+(W_3-W_4)\}
=W_2\{q+(W_3-W_4)\}=0.
\]
If $q+(W_3-W_4)\neq0$ we must have $W_1=W_2=0$, namely $W=W_3E_3+W_4E_4$.
Since $[E_3,E_4]=0$ we get
$$
\gamma(s)=\exp_\mathfrak{osc}(sW).o=\exp_\mathfrak{osc}(s(W_3E_3))\cdot
\exp_\mathfrak{osc}(s(W_4E_4)).o=\exp_\mathfrak{osc}(s(W_3E_3)).o,
$$
which is a geodesic (Reeb flow).

Therefore, we take $q+(W_3-W_4)=0$.
Hence we obtain
\[
W=W_1E_1+W_2E_2+W_3(E_3+E_4)+qE_4.
\]
Thus 
\[
\gamma(s)=\exp_{\mathfrak{osc}}(sW)\cdot
o=\exp_{\mathfrak{osc}}
\left(
s\{W_{\mathfrak{m}}+qE_4\,
\}
\right)\cdot o.
\]
In other words, $\exp_{\mathfrak{osc}}(sW)\cdot
o$ is a homogeneous contact magnetic geodesic 
if and only if $W_{\mathfrak h}=qE_4$. 

\begin{theorem}\label{thm:7.1}
Every homogeneous contact magnetic geodesic starting at 
the origin $o$ of $\mathrm{Nil}_3$ is expressed as
\[
\gamma(s)=\exp_{\mathfrak{osc}}
\left(
s\{W+qE_4\,
\}
\right)\cdot o
\] 
for some $W=W_1E_1+W_2E_2+W_3(E_3+E_4)\in\mathfrak{m}$. 
In case $W_3+q\not=0$, the parametric equation of $\gamma(s)$ is given by 
\begin{align*}
x(s)=&\frac{1}{W_3+q}
(-W_2+W_1\sin((W_3+q)s)+W_2\cos((W_3+q)s)\,),
\\
y(s)=&\frac{1}{W_3+q}
(W_1-W_1\cos((W_3+q)s)+W_2\sin((W_3+q)s)\,),
\\
z(s)=&
\frac{1}{2(W_3+q)^2}\left(
-(W_1^2+W_2^2)\sin((W_3+q)s)+(W_3+q)
(W_1^2+W_2^2+2W_3^2+2qW_3)s\,
\right).
\end{align*}
In case $W_3+q=0$, 
\[
x(s)=W_1\,s,
\quad 
y(s)=W_2\, s,
\quad 
z(s)=-qs.
\]
\end{theorem}
\begin{proof}
In case $W_3+q\not=0$, the homogeneous contact magnetic geodesic 
$\exp_{\mathfrak{osc}}(s(W+q E_4))\cdot o$ is 
computed as 
\[
\exp_{\mathfrak{osc}}(s(W+qE_4))\cdot o=
\exp(s(W+qE_4))\mathsf{M}(0,0,0,-(W_3+q)s).
\]
\end{proof}

Now we arrive at the main result of this article:

\begin{theorem}\label{thm:Main}
Every contact magnetic geodesic 
$\gamma(s)$ starting at the origin in $\mathrm{Nil}_3$ is 
$\mathrm{Osc}$-homogeneous and represented by
\[
\gamma(s)=\exp_{\mathfrak{osc}}(s(W+qE_4))\cdot o
\]
for some $W\in\mathfrak{m}$.
\end{theorem}
\begin{proof}
Let $\gamma(s)$ be a contact magnetic geodesic 
in $\mathrm{Nil}_3$ starting at the origin with initial velocity 
$\dot{\gamma}(0)=aE_1+bE_2+cE_3$ and set
\[
W_1=a,\quad W_2=b,\quad W_3=\cos\theta,
\]
where $\theta$ is the constant contact angle. 
Then the parametric equation of $\gamma(s)$ is 
given by Proposition \ref{prop:mag}. 
Comparing the system \eqref{eq:mag} and 
\eqref{eq:mag0} with Theorem \ref{thm:7.1}, 
one can confirm that $\gamma(s)$ coincides with
\[
\exp_{\mathfrak{osc}}(s(W+qE_4))\cdot o,
\]
where $W=W_1E_1+W_2E_2+W_3(E_3+E_4)\in\mathfrak{m}$.
\end{proof}

\bigskip

{\bf Author contributions.} 
All authors contributed equally in conception, design and preparation of the manuscript.

{\bf Funding.} The first named author is partially supported by JSPS KAKENHI JP23K03081.
The second named author was partially supported by the project (EXCELENT-UAIC), 
code CNFIS-FDI-2025-F-0318 of the University Alexandru Ioan Cuza of Iasi.

\bibliographystyle{amsplain}

\end{document}